\documentclass[12pt,a4paper,dvipdfmx]{article}
\usepackage{amsmath,amssymb,amsthm,enumerate,mathrsfs}
\usepackage[all]{xy}
\usepackage[numbers]{natbib}

\def\A{{\mathcal{A}}}
\def\B{{\mathcal{B}}}
\def\E{{\mathcal{E}}}
\def\I{{\mathcal{I}}}
\def\N{{\mathbb{N}}}
\def\P{{\mathcal{P}}}
\def\R{{\mathbb{R}}}
\def\S{{\mathcal{S}}}
\def\U{{\mathcal{U}}}
\def\V{{\mathcal{V}}}

\theoremstyle{plain}
\newtheorem{thm}{Theorem}[section]

\newtheorem{cor}{Corollary}[section]
\newtheorem{prop}{Proposition}[section]

\theoremstyle{definition}
\newtheorem{ax}{Axiom}[section]
\newtheorem{dfn}{Definition}[section]
\newtheorem{ex}{Example}[section]
\newtheorem{rem}{Remark}[section]

\numberwithin{equation}{section}
\allowdisplaybreaks

\title{Vector Representation of Preferences on $\sigma$-Algebras and Fair Division in Saturated Measure Spaces\thanks{I am grateful to M. Ali Khan, who insisted a more suitable usage on the terminology for axiomatization, and a referee for helpful comments. This research was supported by JSPS KAKENHI Grant Number JP18K01518 from the Ministry of Education, Culture, Sports, Science and Technology, Japan.}}
\date{\today}
\author{Nobusumi Sagara \\
{\small Faculty of Economics, Hosei University} \\[-4pt]
{\small 4342, Aihara, Machida, Tokyo 194--0298, Japan} \\[-4pt]
{\footnotesize e-mail: nsagara@hosei.ac.jp}} 

\begin{document}
\maketitle
\setcounter{page}{0}
\thispagestyle{empty}

\begin{abstract}
The purpose of this paper is twofold. First, we axiomatize preference relations on a $\sigma$-\hspace{0pt}algebra of a saturated measure space represented by a vector measure and furnish a utility representation in terms of a nonadditive measure satisfying the appropriate requirement of continuity and convexity. Second, we investigate the fair division problems in which each individual has nonadditive preferences on a $\sigma$-\hspace{0pt}algebra invoking our utility representation result. We show the existence of individually rational Pareto optimal partitions, Walrasian equilibria, core partitions, and Pareto optimal envy-free partitions. \\


\end{abstract}
\clearpage

\section{Introduction}
Dividing a fixed amount of resources among members of a society to achieve efficiency and fairness is a central theme of social decision making. The first constructive solution to the problem of dividing a ``cake'' fairly was provided by \citet{st48}. Here, cake is simply a metaphor for a heterogeneously divisible commodity, and subsequently, the problem was formulated elegantly in \citet{ds61} as a partitioning problem in a measure space in which the preferences of each individual for pieces of a cake are represented by a nonatomic probability measure. Although the mathematical problem under investigation has attracted considerable attention in recent years, it has a long history. For the historical background, see \cite[Chapter 1]{ba05} and \cite[Chapter 2]{bt96}. Regarding the existence of various solutions for efficiency and fairness and the construction of the protocol/algorithm to obtain fair solutions, satisfactory results have been obtained under the additivity of preferences of each individual; see \cite{ba05,ds61,we85} for the nonconstructive existence of solutions and \cite{bt96,rw98} for the construction of the protocol/algorithm for fair solutions. 

Incidentally, representing preference relations on a $\sigma$-algebra by means of a probability measure means that the utility function inevitably exhibits ``constant marginal utility''. This implies that given an additional element $X$ in a $\sigma$-\hspace{0pt}algebra $\Sigma$, the utility function $\nu$ on $\Sigma$ given by a probability measure satisfies $\nu(A\cup X)-\nu(A)=\nu(X)$ for every $A\in \Sigma$ that is disjoint from $X$, which reveals that the marginal utility is independent of $A$. Obviously, this is a severe restriction on preference relations that is difficult to justify from an economics point of view. Thus, the reasonable conditions under which the preferences of each individual can have nonadditive representations should be addressed for fair division problems. For earlier attempts to axiomatize the preference relations on a finite measure space with their continuous representation by means of a nonadditive measure, see \cite{dm09,sv09}.   

The purpose of this paper is twofold. First, we axiomatize preference relations on a $\sigma$-\hspace{0pt}algebra of a finite measure space represented by a vector measure and furnish a utility representation in terms of a nonadditive measure that satisfies the appropriate requirement of continuity and convexity, which presents a different approach from \cite{dm09,sv09} with a numerical representation of preference relations by means of a nonatomic finite measure. The axioms we introduce here guarantee that the utility functions on a $\sigma$-\hspace{0pt}algebra are continuous quasiconcave transformations of a vector measure with values in a Banach space, which lays the axiomatic foundation for nonadditive utility functions exploited in \cite{hu08,hu11,hs13,sa06} for fair division problems. In particular, if the underlying measure space is ``saturated'' as formulated in \cite{fk02,hk84,ks09}, then such a utility representation is always viable in a separable Banach space along with the Lyapunov convexity theorem in infinite dimensions established in \cite{ks13,ks15,ks16}. This utility representation has a great advantage because the preference relations on a $\sigma$-\hspace{0pt}algebra have a continuous convex extension to the set of measurable functions with values in the unit interval. We characterize the saturation of measure spaces in terms of the continuous extensions of preference relations on $\sigma$-algebras.  

Second, we investigate the fair division problems in which each individual has nonadditive preferences on a $\sigma$-\hspace{0pt}algebra invoking our utility representation result. If the preferences of each individual are represented by a nonatomic probability measure, then the classical Lyapunov convexity theorem \cite{ly40} guarantees the compactness and convexity of the utility possibility set and thereby makes it possible to establish the existence of solutions with respect to efficiency and fairness; see \cite{ba05,ds61,we85}. In consideration of nonadditive utility functions on $\sigma$-\hspace{0pt}algebras, such favorable properties are no longer valid. For the clarification of the role of the nonatomicity hypothesis in the fair division problems with additive utility functions on $\sigma$-\hspace{0pt}algebras, see the survey article \cite{sa11}. 

To overcome the difficulty with nonadditive utility functions on $\sigma$-\hspace{0pt}algebras, as in \cite{ak95,da01,hs13}, we subsume partitions of an economy in which the preferences of each individual are represented by a continuous transformation of a vector measure into allocations of its extended economy with the commodity space of $L^\infty$ in which the preferences of each individual on a $\sigma$-algebra are continuously extended to the subset of functions in $L^\infty$ with values in the unit interval. We clarify the role of saturation of the underlying measure space to formulate the indifference relation for each individual between partitions and allocations, which provides another characterization of saturation. Under the saturation hypothesis, we show the existence of individually rational Pareto optimal partitions without any convexity assumption on the preferences of each individual and the existence of Walrasian equilibria, core partitions, and Pareto optimal envy-free partitions under the convexity assumption. 

The paper is organized as follows. After the brief introduction of saturated measure spaces and the Lyapunov convexity theorem in separable Banach spaces in Section 2, the axiomatization for the preference relations on a $\sigma$-\hspace{0pt}algebra is given in Section 3. The fair division problems are framed in Section 4 and the existence results on fair partitions are provided under the saturation hypothesis. Lastly, an open question is stated in Section 5 as a concluding remark.

\section{Preliminaries}
\subsection{Lyapunov Convexity Theorem in Saturated Measure Spaces}
Throughout this paper, we always assume that $(\Omega,\Sigma,\mu)$ is a finite measure space. A measure space $(\Omega,\Sigma,\mu)$ is said to be \textit{essentially countably generated} if its $\sigma$-\hspace{0pt}algebra can be generated by a countable number of subsets together with the null sets; $(\Omega,\Sigma,\mu)$ is said to be \textit{essentially uncountably generated} whenever it is not essentially countably generated. Let $\Sigma_X=\{ X\cap A\mid A\in \Sigma \}$ be the $\sigma$-\hspace{0pt}algebra restricted to $X\in \Sigma$. Denote by $L^1(X,\Sigma_X,\mu)$ the space of $\mu$-\hspace{0pt}integrable functions on the measurable space $(X,\Sigma_X)$ whose elements are restrictions of functions in $L^1(\Omega,\Sigma,\mu)$ to $X$. An equivalence relation $\equiv$ on $\Sigma$ is given by $A\equiv B \Longleftrightarrow \mu(A\triangle B)=0$, where $A\triangle B$ is the symmetric difference of $A$ and $B$ in $\Sigma$. The collection of equivalence classes is denoted by $\Sigma(\mu)=\Sigma/\equiv$ and its generic element $\widehat{A}$ is the equivalence class of $A\in \Sigma$. The metric $\rho$ on $\Sigma(\mu)$ is defined by $\rho(\widehat{A},\widehat{B})=\mu(A\triangle B)$. Then $(\Sigma(\mu),\rho)$ is a complete metric space (see \cite[Lemma 13.13]{ab06}) and $(\Sigma(\mu),\rho)$ is separable if and only if $L^1(\Omega,\Sigma,\mu)$ is separable (see \cite[Lemma 13.14]{ab06}). The \textit{density} of $(\Sigma(\mu),\rho)$ is the smallest cardinal number of the form $|\U|$, where $\U$ is a dense subset of $\Sigma(\mu)$. 

\begin{dfn}
A finite measure space $(\Omega,\Sigma,\mu)$ is \textit{saturated} if $L^1(X,\Sigma_X,\mu)$ is nonseparable for every $X\in \Sigma$ with $\mu(X)>0$. We say that a finite measure space has the \textit{saturation property} if it is saturated.
\end{dfn}

Saturation implies nonatomicity and several equivalent definitions for saturation are known; see \cite{fk02,fr12,hk84,ks09}. One of the simple characterizations of the saturation property is as follows. A finite measure space $(\Omega,\Sigma,\mu)$ is saturated if and only if $(X,\Sigma_X,\mu)$ is essentially uncountably generated for every $X\in \Sigma$ with $\mu(X)>0$. The saturation of finite measure spaces is also synonymous with the uncountability of the density of $\Sigma_X(\mu)$ for every $X\in \Sigma$ with $\mu(X)>0$; see \cite[331Y(e)]{fr12}. An inceptive notion of saturation already appeared in \cite{ka44,ma42}. 

Let $E$ be a Banach space. For a vector measure $m:\Sigma\to E$, a set $N\in \Sigma$ is said to be \textit{$m$-null} if $m(A\cap N)=0$ for every $A\in \Sigma$. A vector measure $m:\Sigma\to E$ is said to be \textit{$\mu$-\hspace{0pt}continuous} (or \textit{absolutely continuous} with respect to $\mu$) if every $\mu$-null set is $m$-null. Let $\S:=\{f\in L^\infty(\Omega,\Sigma,\mu)\mid 0\le f\le 1 \}$, which is a weakly$^*\!$ compact, convex subset of $L^\infty(\Omega,\Sigma,\mu)$, and define $m(\S):=\{ \int fdm\in E\mid f\in \S \}$. Since the integration operator $T_m:L^\infty(\Omega,\Sigma,\mu)\to E$ defined by $T_m(f):=\int fdm$ is continuous with respect to the weak$^*\!$ topology of $L^\infty(\Omega,\Sigma,\mu)$ and the weak topology of $E$ (see \cite[Lemma IX.1.3]{du77}), the set $m(\S)$ is weakly compact and convex in $E$. 

The saturation property and the Lyapunov convexity theorem in separable Banach spaces are a major apparatus in this paper. 

\begin{prop}[\citet{ks13}]
\label{lyap}
Let $E$ be a separable Banach space. If $(\Omega,\Sigma,\mu)$ is saturated, then for every $\mu$-continuous vector measure $m:\Sigma\to E$, its range $m(\Sigma)$ is weakly compact and convex with $m(\Sigma)=m(\S)$. Conversely, if every $\mu$-continuous vector measure $m:\Sigma\to E$ has the weakly compact convex range, then $(\Omega,\Sigma,\mu)$ is saturated whenever $E$ is infinite dimensional. 
\end{prop}

\begin{rem}
The significance of the saturation property lies in the fact that it is necessary and sufficient for the weak compactness and the convexity of the Bochner integral of a multifunction as well as the Lyapunov convexity theorem; see \cite{ks13,ks14,ks15,ks16,po08,sy08}. For the further generalization of Proposition \ref{lyap} to nonseparable locally convex spaces, see \cite{gp13,ks15,ks16,sa17}. Another intriguing characterization of saturation in terms of the existence of Nash equilibria in large games is found in \cite{ks09}. 
\end{rem}

\section{Vector Representation of Preference Relations on $\sigma$-\hspace{0pt}Algebras}
\subsection{Axioms for Preference Relations}
Let $(\Omega,\Sigma,\mu)$ be a finite measure space. A binary relation on the $\sigma$-algebra $\Sigma$ is a subset of the product space $\Sigma\times \Sigma$. The \textit{preference relation} $\succsim$ on $\Sigma$ is a complete transitive binary relation on $\Sigma$. We denote by $A\succsim B$ the relation $(A,B)\in{}\succsim$. The indifference and strict preference relations are defined respectively by $A\sim B\Longleftrightarrow A\succsim B$ and $B\succsim A$ and by $A\succ B \Longleftrightarrow A\succsim B$ and $A\not\sim B$. A real-valued function $\nu:\Sigma\to \R$ is called a \textit{utility function} representing $\succsim$ if for every $A,B\in \Sigma$: $\nu(A)\ge \nu(B) \Longleftrightarrow A\succsim B$. A set function $\nu:\Sigma\to \R$ is called a \textit{nonadditive measure} if $\nu(\emptyset)=0$. 

\begin{ax}[vector representation]
\label{ax1}
There exist a Banach space $E$ and a $\mu$-\hspace{0pt}continuous vector measure $m:\Sigma\to E$ such that $m(A)=m(B)$ implies $A\sim B$.  
\end{ax}
\noindent
The axiom enables one to define the preference relation $\mathscr{R}$ on $m(\Sigma)$ by 
\begin{equation}
\label{eq1}
\forall x,y\in m(\Sigma): x\,\mathscr{R}\,y \stackrel{\text{def}}{\Longleftrightarrow} A\succsim B \text{ with $x=m(A)$ and $y=m(B)$}.
\end{equation}
A preference relation on $\Sigma$ satisfying Axiom \ref{ax1} is said to admit a \textit{vector representation} in $E$. A vector representation for $\succsim$ is not unique because any scalar multiplication of $m$ is consistent with Axiom \ref{ax1}. In addition, the observation that $\succsim$ may admit a vector representation in another Banach space motivates one to introduce the following axiom. (In what follows, we assume that $\succsim$ admits a vector representation in a Banach space $E$ via a $\mu$-continuous vector measure $m:\Sigma\to E$.) 

\begin{ax}[commutativity]
\label{ax2}
If $\succsim$ admits another vector representation in a Banach space $F$ via a $\mu$-continuous vector measure $n:\Sigma\to F$, then there exists a unique continuous linear operator $T:E\to F$ such that $n=T\circ m$. 
\end{ax}
\noindent
By the symmetric treatment of $E$ and $F$, the axiom guarantees that there exists a unique continuous linear operator $U:F\to E$ such that $m=U\circ n$. Then 
$(U\circ T)m=m$ and $(T\circ U)n=n$. The situation is illustrated in the diagram below: 
$$
\xymatrix{
& \ E  \ar[dd]_{T} \ar@{<-}@<1.0ex>[dd]^{U}  \\
\Sigma \ar[ur]^m \ar[dr]_n   \\
& \ F  
}
$$
A preference relation on $\Sigma$ satisfying Axiom \ref{ax2} is said to admit \textit{commutativity} in $E$. Commutativity guarantees that the vector representation for $\succsim$ is unique up to the equivalence in the following sense. Let $\V(\Sigma,\mu)$ be the set of $\mu$-continuous vector measures on $\Sigma$ with values in any Banach space. A typical element in $\V(\Sigma,\mu)$ is denoted by $(m,E)$ to distinguish explicitly the range spaces of vector measures. Define the equivalence relation on $\V(\Sigma,\mu)$ by $(m,E)\simeq(n,F)\stackrel{\text{def}}\Longleftrightarrow$ there exists a unique pair of continuous linear operators $T:E\to F$ and $U:F\to E$ such that $n=T\circ m$ and $m=U\circ n$. As defined in \eqref{eq1}, the commutativity for $\succsim$ yields an induced preference relation $\mathscr{Q}$ on $n(\Sigma)$. Thus, for every $x,y\in m(\Sigma)$: $x\,\mathscr{R}\,y \Longleftrightarrow Tx\,\mathscr{Q}\,Ty$ and for every $v,w\in n(\Sigma)$: $v\,\mathscr{Q}\,w \Longleftrightarrow Uv\,\mathscr{R}\,Uw$.  

\begin{ax}[monotonicity]
\label{ax3}
For every $A,B\in \Sigma$ with $A\supset B$ and $m(A\setminus B)\ne 0$, we have $A\,\succ\,B$. 
\end{ax}
\noindent
The axiom is equivalent to the monotonicity of $\mathscr{R}$ on $m(\Sigma)$: for every $x,y\in m(\Sigma)$ with $x+y\in m(\Sigma)$ and $y\ne 0$, we have $x+y\,\mathscr{R}\,x$, but not $x\,\mathscr{R}\,x+y$. A preference relation on $\Sigma$ satisfying Axiom \ref{ax3} is said to admit a \textit{monotone representation} in $E$. 

\begin{ax}[continuity]
\label{ax4}
For every $A,B\in \Sigma$: $A^k\succsim B$ for each $k\in \N$ with $m(A^k)\to m(A)$ as $k\to \infty$ implies $A\succsim B$ and $A\succsim B^k$ for each $k\in \N$ with $m(B^k)\to m(B)$ implies $A\succsim B$. 
\end{ax}
\noindent
The axiom is equivalent to the continuity of $\mathscr{R}$ on $m(\Sigma)$: for every $x\in m(\Sigma)$, both the upper contour set $\{ y\in m(\Sigma)\mid y\,\mathscr{R}\,x \}$ and the lower contour set $\{ y\in m(\Sigma)\mid x\,\mathscr{R}\,y \}$ are closed in $m(\Sigma)$. A preference relation on $\Sigma$ satisfying Axiom \ref{ax4} is said to admit a \textit{continuous representation} in $E$. 

\begin{ax}[convexity]
\label{ax5}
Suppose that $m(\Sigma)$ is convex. For every $A\in \Sigma$, the set $m(\{ B\in \Sigma\mid B\,\succsim\,A \})$ is convex. 
\end{ax}
\noindent
The axiom is equivalent to the convexity of $\mathscr{R}$ on $m(\Sigma)$: for every $x\in m(\Sigma)$, the upper contour set $\{ y\in m(\Sigma)\mid y\,\mathscr{R}\,x \}$ is convex. A preference relation on $\Sigma$ satisfying Axiom \ref{ax5} is said to admit a \textit{convex representation} in $E$. Note that the convexity of $m(\Sigma)$ is not automatic even when $m$ is nonatomic because of the well-known failure of the Lyapunov convexity theorem in infinite dimensions (for counterexamples in which the Lyapunov convexity theorem fails in infinite dimensions, see \cite[Examples IX.1.1 and IX.1.2]{du77}).

\subsection{Utility Functions on $\Sigma$}
Axioms \ref{ax1}, \ref{ax2}, and \ref{ax4} guarantee the unique representation of $\succsim$ in terms of a nonadditive measure on $\Sigma$ given by a continuous transformation of a vector measure. If, moreover, Axiom \ref{ax5} is imposed, then the representation is a quasiconcave transformation of a vector measure.   

\begin{thm}
\label{thm1}
A preference relation $\succsim$ on $\Sigma$ admits a continuous (and convex) representation in a separable Banach space $E$ if and only if there exist a $\mu$-continuous vector measure $m:\Sigma\to E$ (with $m(\Sigma)$ being convex) and a continuous (and quasiconcave) function $\varphi:m(\Sigma)\to \R$ with $\varphi(0)=0$ such that:
\begin{equation}
\label{eq2}
\forall A,B\in \Sigma: A\succsim B \Longleftrightarrow \varphi(m(A))\ge \varphi(m(B)).
\end{equation}
If, furthermore, $\succsim$ admits commutativity, then the continuous (and convex) representation via $(m,E)\in \V(\Sigma,\mu)$ is unique up to the equivalence classes in $\V(\Sigma,\mu)/\simeq$. 
\end{thm}

\begin{proof}
Suppose that $\succsim$ admits a continuous (and convex) representation in a separable Banach space $E$. Then there exists a $\mu$-continuous vector measure $m:\Sigma\to E$ (with $m(\Sigma)$ being convex) such that \eqref{eq1} holds. Since the induced preference relation $\mathscr{R}$ on $m(\Sigma)$ is continuous (and convex) and $E$ is separable, by the celebrated theorem of \citet{de64}, there exists a continuous (and quasiconcave) function $\varphi:m(\Sigma)\to \R$ representing the preference relation $\mathscr{R}$ on $m(\Sigma)$, that is, $x\,\mathscr{R}\,y\Longleftrightarrow \varphi(x)\ge \varphi(y)$. Therefore, $\varphi\circ m$ is a utility function on $\Sigma$ representing the preference relation $\succsim$ on $\Sigma$. By replacing $\varphi(x)$ with $\varphi(x)-\varphi(0)$ if necessary, without loss of generality one may assume that $\varphi(0)=0$. The converse implication is obvious without assuming the separability of $E$.  
\end{proof}

Denote by $L^1_E(\Omega,\Sigma,\mu)$ the space of $E$-\hspace{0pt}valued Bochner integrable functions on $\Omega$. A Banach space $E$ has the \textit{Radon--Nikodym property} (\textit{RNP}) for a finite measure space $(\Omega,\Sigma,\mu)$ if for every $\mu$-\hspace{0pt}continuous vector measure $m:\Sigma\to E$ of bounded variation, there exists $u\in L^1_E(\Omega,\Sigma,\mu)$ such that $m(A)=\int_Aud\mu$ for every $A\in \Sigma$. Note that for every Banach space $E$, given a Bochner integrable function $u\in L^1_E(\Omega,\Sigma,\mu)$, the vector measure $m^u:\Sigma\to E$ defined by the indefinite integral $m^u(A):=\int_Aud\mu$ for $A\in \Sigma$ is of bounded variation; see \cite[Theorem II.2.4]{du77}. This observation yields the following result. 

\begin{cor}
A preference relation $\succsim$ on $\Sigma$ admits a continuous (and convex) representation in a separable Banach space $E$ with the RNP for $(\Omega,\Sigma,\mu)$ if and only if there exist a Bochner integrable function $u\in L^1_E(\Omega,\Sigma,\mu)$ (with $m^u(\Sigma)$ being convex) and a continuous (and quasiconcave) function $\varphi:m^u(\Sigma)\to \R$ with $\varphi(0)=0$ such that:
$$
\forall A,B\in \Sigma: A\succsim B \Longleftrightarrow \varphi\left( \int_Au(\omega)d\mu \right)\ge \varphi\left( \int_Bu(\omega)d\mu \right).
$$
\end{cor}

\begin{ex}
Let $(\Omega,\Sigma,\mu)$ be a nonatomic finite measure space and $f\in L^1(\Omega,\Sigma,\mu)$ be a nonnegative integrable function. If $\succsim$ is such that for every $A,B\in \Sigma$: $A\succsim B \stackrel{\text{def}}{\Longleftrightarrow} \int_Afd\mu\ge \int_Bfd\mu$, then $\succsim$ admits a vector representation in the real field $\R$ via the $\mu$-continuous nonatomic finite measure $\nu$ defined by $\nu(A):=\int_Afd\mu$ and satisfies Axioms \ref{ax1}, \ref{ax3}, \ref{ax4}, and \ref{ax5}. Conversely, if $\succsim$ admits a vector representation in $\R$ via the $\mu$-continuous $\sigma$-finite measure $\nu$ and satisfies Axioms \ref{ax1}, \ref{ax3}, \ref{ax4}, and \ref{ax5}, then by Theorem \ref{thm1}, there exists a continuous quasiconcave function $\varphi:\nu(\Sigma)\to \R$ with $\varphi(0)=0$ such that for every $A,B\in \Sigma$: $A\succsim B \Longleftrightarrow \varphi(\int_Afd\mu)\ge \varphi(\int_Bfd\mu)$, where $f\in L^1(\Omega,\Sigma,\mu)$ is a Radon--Nikodym derivative of $\nu$. In this specific case, $\succsim$ admits a continuous convex representation with a quasiconcave integral transformation on $\Sigma$.
\end{ex}

\begin{rem}
Note that the proof of Theorem \ref{thm1} demonstrates that if $(\Omega,\Sigma,\mu)$ is a nonatomic finite measure space, then $\succsim$ admits a continuous (and convex) representation in $\R^l$ if and only if there exist a $\mu$-continuous vector measure $m:\Sigma\to \R^l$ with its range $m(\Sigma)$ being compact and convex and a continuous (and quasiconcave) function $\varphi:m(\Sigma)\to \R$ with $\varphi(0)=0$ satisfying \eqref{eq2}. This is an immediate consequence of the classical Lyapunov convexity theorem in view of the nonatomicity of $m$. 
\end{rem}

\subsection{Continuous Extensions to $\S$}
Denote by $\chi_A$ the characteristic function of $A\in \Sigma$. Then $\chi_A\in \S$ for every $A\in \Sigma$. A preference relation ${\succsim}^{\diamond}$ on $\S$ is a complete transitive binary relation on $\S$. A preference relation ${\succsim}^{\diamond}$ is said to be \textit{continuous} if for every $f\in \S$ both the upper contour set $\{ g\in \S\mid g\,{\succsim}^{\diamond}\,f \}$ and the lower contour set $\{ g\in \S\mid f\,{\succsim}^{\diamond}\,g \}$ are weakly$^*\!$ closed in $\S$; ${\succsim}^{\diamond}$ is said to be \textit{convex} if for every $f\in \S$ the upper contour set $\{ g\in \S\mid g\,{\succsim}^{\diamond}\,f \}$ is convex; ${\succsim}^{\diamond}$ is said to be \textit{monotone} if $f\ge g$ with $f\ne g$ imply $f\,\succ\,g$; ${\succsim}^{\diamond}$ is called an \textit{extension} of the preference relation $\succsim$ on $\Sigma$ if for every $A,B\in \Sigma$: $\chi_A\,{\succsim}^{\diamond}\,\chi_B \Longleftrightarrow A\succsim B$. 

If $\succsim$ is represented by a utility function of the form $\varphi\circ m$ on $\Sigma$, where $m$ is a $\mu$-\hspace{0pt}continuous vector measure with values in a Banach space $E$ (with $m(\Sigma)$ being convex) and $\varphi$ is a continuous (and quasiconcave) function on $m(\Sigma)$ with a continuous (and quasiconcave) extension $\varphi^\diamond$ to $m(\S)$, then the continuous (and convex) extension $\succsim^\diamond$ of $\succsim$ to $\S$ is given by: 
\begin{equation}
\label{eq3}
\forall f,g\in \S: f\,{\succsim}^{\diamond}g \stackrel{\text{def}}\Longleftrightarrow \varphi^\diamond\left( \int_\Omega f(\omega)dm \right)\ge \varphi^\diamond\left( \int_\Omega g(\omega)dm \right).
\end{equation}
The continuous (and convex) extension $\succsim^\diamond$ defined in this way is called a \textit{$($quasiconcave$)$ integral transformation} on $\S$. Then the utility function $\nu^\diamond:\S\to \R$ defined by $\nu^\diamond(f):=\varphi^\diamond(\int fdm)$ is weakly$^*\!$ continuous (and quasiconcave) on $S$ in view of the fact that the integration operator $T_m$ is continuous with respect to the weak$^*\!$ topology of $L^\infty(\Omega,\Sigma,\mu)$ and the weak topology of $E$, as noted above. 

\begin{thm}
\label{thm2}
Let $(\Omega,\Sigma,\mu)$ be a saturated finite measure space. Then every preference relation $\succsim$ on $\Sigma$ with a continuous representation in a separable Banach space has a continuous extension $\succsim^\diamond$ to $\S$ with an integral transformation such that for every $f\in \S$ there exists $A\in \Sigma$ satisfying $f\sim^\diamond \chi_A$.  
\end{thm}

\begin{proof}
Let $\succsim$ be a preference relation on $\Sigma$ that admits a continuous representation in a separable Banach space $E$. By Theorem \ref{thm1}, $\succsim$ has a utility function of the form $\varphi\circ m$, where $m$ is a $\mu$-continuous vector measure with values in a separable Banach space $E$ and $\varphi$ is a continuous function on $m(\Sigma)$. In view of Proposition \ref{lyap}, $m(\Sigma)=m(\S)$, and hence, the desired extension ${\succsim}^{\diamond}$ of $\succsim$ to $\S$ is given by the formula \eqref{eq3} with $\varphi=\varphi^\diamond$. 
\end{proof}

The converse of Theorem \ref{thm2} is true in the following sense, which provides another characterization of saturation. 

\begin{thm}
\label{thm3}
A finite measure space $(\Omega,\Sigma,\mu)$ is saturated if every preference relation $\succsim$ on $\Sigma$ with a continuous representation in an infinite-\hspace{0pt}dimensional Banach space and its continuous extension $\succsim^\diamond$ to $\S$ with an integral transformation are such that for every $f\in \S$ there exists $A\in \Sigma$ satisfying $f\sim^\diamond\chi_A$.  
\end{thm}

\begin{proof}
Suppose that $(\Omega,\Sigma,\mu)$ is not saturated. Then for every infinite-dimensional Banach space $E$, there exists a $\mu$-continuous vector measure $m:\Sigma\to E$ such that its range $m(\Sigma)$ is not convex; see \cite[Lemma 4]{po08} and \cite[Remark 1]{sy08}. Then there exist $x,y\in m(\Sigma)$ and $t\in (0,1)$ such that $tx+(1-t)y\not\in m(\Sigma)$. Let $C,D\in \Sigma$ be such that $x=m(C)$ and $y=m(D)$ and define $\hat{f}:=t\chi_C+(1-t)\chi_D\in \S$. Let $\varphi^\diamond:m(\S)\to \R$ be a continuous function with $\varphi^\diamond(0)=0$ that attains the maximum at the unique point $\hat{x}:=\int \hat{f}dm\in m(\S)$ (such a function can be simply given by, for instance,  the strictly concave continuous function $\varphi^\diamond(z)=\| \hat{x}\|-\| \hat{x}-z \|$). Define the preference relation $\succsim^\diamond$ on $\S$ via \eqref{eq3} and denote its restriction to $\Sigma$ by $\succsim$. Then $\succsim$ admits a continuous representation in the Banach space $E$ by Theorem \ref{thm1} (the separability of $E$ is unnecessary here) and $\succsim^\diamond$ is a continuous extension of $\succsim$ with an integral transformation. Hence, there exists $A\in \Sigma$ such that $\hat{f}\sim^\diamond\chi_A$. This means that $\varphi^\diamond(\hat{x})=\varphi^\diamond(m(A))$, and hence, $\hat{x}=tx+(1-t)y=m(A)\in m(\Sigma)$, a contradiction. 
\end{proof}

\begin{rem}
The continuous extension of nonadditive measures on $\Sigma$ to $\S$ with an integral transformation explored in Theorem \ref{thm2} is a further improvement of \cite{hs12} under the saturation hypothesis incorporating infinite dimensional vector representation. An alternative formulation of the convexity of preference relations on $\sigma$-\hspace{0pt}algebras and their representation in terms of a nonadditive measure with the quasiconcavity-\hspace{0pt}like property based on the Lyapunov convexity theorem in finite dimensions were studied in \cite{sv09}. For continuous concave extensions of nonadditive measures to $L^\infty(\Omega,\Sigma,\mu)$ based on Choquet integrals and the support functionals of the cores of exact games, see \cite{sa13}.
\end{rem}

\section{Fair Division in Saturated Measure Spaces}
\subsection{Extensions of Economies}
The problem of dividing a heterogeneous commodity among a finite number of individuals is formulated as partitioning in a finite measure space $(\Omega,\Sigma,\mu)$. Here, the set $\Omega$ is a heterogeneously divisible commodity, the $\sigma$-\hspace{0pt}algebra $\Sigma$ of subsets of $\Omega$ denotes the collection of possible pieces of $\Omega$, and the finite measure $\mu$ describes an objective evaluation for a cardinal attribute of each piece in $\Sigma$. There are $n$ individuals and the set of individuals is denoted by $I=\{ 1,2,\dots,n \}$, each of whom is  indexed by $i\in I$, whose preference relation $\succsim_i$ is defined on a common consumption set $\Sigma$. Each individual possesses an initial endowment $\Omega_i\in \Sigma$, for which $\Omega_i\cap \Omega_j=\emptyset$ for each $i\ne j$ and $\bigcup_{i=1}^n\Omega_i=\Omega$. A \textit{partition} of $\Omega$ is an ordered $n$-\hspace{0pt}tuple $(A_1,\dots,A_n)$ of mutually disjoint sets $A_1,\dots,A_n$ in $\Sigma$ whose union is $\Omega$, where each $A_i$ is a share of the divisible commodity $\Omega$ given to individual $i\in I$. Thus, $(\Omega_1,\dots,\Omega_n)$ is an \textit{initial partition} of $\Omega$. An \textit{economy} $\E=\{ (\Omega,\Sigma,\mu),(\succsim_i,\Omega_i)_{i=1}^n \}$ consists of the commodity space $(\Omega,\Sigma,\mu)$, the common consumption set $\Sigma$, on which a preference relation $\succsim_i$ is defined for each $i\in I$, and the initial endowment $\Omega_i$ of each individual. 

Following \cite{ak95,da01,hs12}, we subsume economies with commodity space $(\Omega,\Sigma,\mu)$ into ones with commodity space $L^\infty(\Omega,\Sigma,\mu)$. If each $\succsim_i^\diamond$ is a preference relation on $\S$ that is an extension of $\succsim_i$ given an economy $\E$, then $\E^\diamond=\{ L^\infty(\Omega,\Sigma,\mu),\S,({\succsim_i^\diamond},\chi_{\Omega_i})_{i=1}^n \}$ is called an \textit{extended economy} of $\E$, which consists of the commodity space $L^\infty(\Omega,\Sigma,\mu)$, the common consumption set $\S$, on which a preference relation $\succsim_i^\diamond$ is defined for each $i\in I$, and the initial endowment $\chi_{\Omega_i}\in \S$ of each individual. If $\succsim_i$ is represented by a utility function of the form $\varphi_i\circ m_i$ on $\Sigma$, where $m_i$ is a $\mu$-continuous vector measure and $\varphi_i$ is a continuous (and quasiconcave) function on $m(\Sigma)$ (with $m(\Sigma)$ being convex) that has an continuous (and quasiconcave) extension $\varphi_i^\diamond$ to $m_i(\S)$, then a continuous (and convex) extension $\succsim_i^\diamond$ to $\S$ with a (quasiconcave) integral transformation of the preference relation $\succsim_i$ on $\Sigma$ is given by the formula \eqref{eq3}. In this case, $\E^\diamond$ is called an extended economy of $\E$ with a \textit{(quasiconcave) integral transformation}. An ordered $n$-\hspace{0pt}tuple $(f_1,\dots,f_n)$ of functions in $L^\infty(\Omega,\Sigma,\mu)$ is called an \textit{allocation} of $\chi_\Omega$ if $\sum_{i=1}^nf_i(\omega)=1$ for every $\omega\in \Omega$ and $f_i\in \S$ for each $i\in I$. Every $A\in \Sigma$ is identified with its characteristic function $\chi_A$ in $\S$. Note that $(A_1,\dots,A_n)$ is a partition of $\Omega$ if and only if $(\chi_{A_1},\dots,\chi_{A_n})$ is an allocation of $\chi_\Omega$. 

\begin{dfn}
An allocation $(f_1,\dots,f_n)$ of $\chi_\Omega$ is said to be:
\begin{enumerate}[(i)]
\item \textit{individually rational} if $f_i\succsim_i^\diamond \chi_{\Omega_i}$ for each $i\in I$; 
\item \textit{Pareto optimal} if there is no allocation $(g_1,\dots,g_n)$ such that $g_i\succsim_i^\diamond f_i$ for each $i\in I$ and $g_j\succ_j^\diamond f_j$ for some $j\in I$;
\item \textit{envy-free} if $f_j\succsim_i^\diamond f_i$ for each $i,j\in I$.
\end{enumerate}
\end{dfn}

A nonempty subset of $I$ is called a \textit{coalition}. The family of coalitions is denoted by $\I$ with its typical element denoted by $S\in \I$. For a coalition $S\in \I$, an allocation $(f_1,\dots,f_n)$ of $\chi_\Omega$ is called an \textit{$S$-\hspace{0pt}allocation} if $\sum_{i\in S}f_i=\sum_{i\in S}\chi_{\Omega_i}$. Then an $I$-allocation is nothing but an allocation of $\chi_\Omega$. 

\begin{dfn}
A coalition $S\in \I$ is said to \textit{improve upon} an allocation $(f_1,\dots,f_n)$ of $\chi_\Omega$ if there exists an $S$-allocation  $(g_1,\dots,g_n)$ such that $g_i\,\succ_i^\diamond\,f_i$ for each $i\in S$. An allocation of $\chi_\Omega$ that cannot be improved upon by any coalition is called a \textit{core allocation}.
\end{dfn}

Let $\mathit{ba}(\Sigma,\mu)$ be the space of finitely additive signed measures vanishing on $\mu$-null sets, which is the price space for the heterogeneously divisible commodity $(\Omega,\Sigma,\mu)$. A price $p\in \mathit{ba}(\Sigma,\mu)$ is said to be \textit{positive} if $p(A)\ge 0$ for every $A\in \Sigma$ and $p\ne 0$. The dual space of $L^\infty(\Omega,\Sigma,\mu)$ is $\mathit{ba}(\Sigma,\mu)$, whose duality relation is denoted by $\langle p,f \rangle=\int fdp$ for $p\in \mathit{ba}(\Sigma,\mu)$ and $f\in L^\infty(\Omega,\Sigma,\mu)$. Given price $p\in \mathit{ba}(\Sigma,\mu)$, the \textit{budget set} of each individual is given by $B_i^\diamond(p,\chi_{\Omega_i}):=\{ f\in \S \mid \langle p,f \rangle\le \langle p,\chi_{\Omega_i} \rangle \}$.    

\begin{dfn}
A price-allocation pair $(p,(f_1,\dots,f_n))$ is called a \textit{Walrasian equilibrium} if $f_i$ is a maximal element for $\succsim_i$ on $B_i^\diamond(p,\chi_{\Omega_i})$ for each $i\in I$. 
\end{dfn}

Denote by $\P$ the set of partitions of $\Omega$, by $\P_\mathrm{IR}$ the set of individually rational partitions, by $\P_\mathrm{PO}$ the set of Pareto optimal partitions, by $\P_\mathrm{EF}$ the set of envy-free partitions, by $\P_\mathrm{C}$ the set of core partitions, and by $\P_\mathrm{WE}$ the set of Walrasian equilibrium partitions respectively in $\E$. Then the inclusion $\P_\mathrm{WE}\subset \P_\mathrm{C}\subset \P_\mathrm{IR}\cap\P_\mathrm{PO}$ is trivial by definition. Denote by $\A$ the set of partitions of $\chi_\Omega$, by $\A_\mathrm{IR}$ the set of individually rational partitions, by $\A_\mathrm{PO}$ the set of Pareto optimal partitions, by $\A_\mathrm{EF}$ the set of envy-free partitions, by $\A_\mathrm{C}$ the set of core allocations, and by $\A_\mathrm{WE}$ the set of Walrasian equilibrium allocations respectively. Then the inclusion $\A_\mathrm{WE}\subset \A_\mathrm{C}\subset \A_\mathrm{IR}\cap \A_\mathrm{PO}$ is trivial by definition. For every economy $\E$ and its extended economy $\E^\diamond$, the inclusions $\P\subset \A$, $\P_\mathrm{IR}\subset \A_\mathrm{IR}$, and $\P_\mathrm{EF}\subset \A_\mathrm{EF}$ are automatic.

In what follows, we assume that each $\succsim_i$ admits a continuous representation in a separable Banach space $E_i$. In view of Theorem \ref{thm1}, each $\succsim_i$ is represented by a utility function $\varphi_i\circ m_i$ on $\Sigma$, where $m_i:\Sigma\to E_i$ is a $\mu$-continuous vector measure and $\varphi_i:m_i(\Sigma)\to \R$ is a continuous function. The utility function of this form means that each individual evaluates any measurable subset $A$ of the heterogeneously divisible commodity $(\Omega,\Sigma,\mu)$ in terms of the subjective cardinal attribute $m_i(A)$ with a vector representation in $E_i$. When $E_i=\R^{l_i}$, there are $l_i$ cardinal attributes for the heterogeneously divisible commodity that is significant to individual $i$, which is the specification modelled in \cite{hu08,hu11,hs12,sa06}.

\begin{rem}
Since an extended economy $\E^\diamond$ of $\E$ is nothing but a standard exchange economy with commodity space $L^\infty(\Omega,\Sigma,\mu)$, the existence of Walrasian equilibria for $\E^\diamond$ simply follows from exactly the same argument as in \cite{be72} if one imposes Mackey continuity and convexity on each $\succsim_i^\diamond$. In the proof of Theorem \ref{thm6} below, we demonstrate the existence of Walrasian equilibria for $\E^\diamond$ in a more specific setting. 
\end{rem}

\subsection{A Characterization of Saturation}
The power of saturation is exemplified in the indifference relation between partitions in $\E$ and allocations in the extended economy $\E^\diamond$ of the original economy $\E$. The proof presented here is based on the Lyapunov convexity theorem in separable Banach spaces (Proposition \ref{lyap}) and the extreme point argument in the weakly$^*\!$ compact subset of $L^\infty$ inspired by a functional analytic approach originated in \cite{li66} and then explored in \cite{ak95,hs12}.    

\begin{thm}
\label{thm4}
Let $(\Omega,\Sigma,\mu)$ be a saturated finite measure space. Then every economy $\E=\{ (\Omega,\Sigma,\mu),(\succsim_i,\Omega_i)_{i=1}^n \}$ for which each $\succsim_i$ admits a continuous representation in a separable Banach space has its extended economy $\E^\diamond=\{ L^\infty(\Omega,\Sigma,\mu),\S,(\succsim_i^\diamond,\chi_{\Omega_i})_{i=1}^n \}$ with an integral transformation such that for every allocation $(f_1,\dots,f_n)$ in $\E^\diamond$, there exists a partition $(A_1,\dots,A_n)$ in $\E$ satisfying $f_i\sim_i^\diamond\chi_{A_i}$ for each $i\in I$. 
\end{thm}

\begin{proof}
Let $\E=\{ (\Omega,\Sigma,\mu),(\succsim_i,\Omega_i)_{i=1}^n \}$ be an economy satisfying the hypothesis of the theorem. In view of Theorem \ref{thm1}, each $\succsim_i$ is represented by a utility function of the form $\varphi_i\circ m_i$, where $m_i$ is a $\mu$-continuous vector measure with values in a separable Banach space $E_i$ and $\varphi_i$ is a continuous function on $m_i(\Sigma)$. Since $m_i(\Sigma)=m_i(\S)$ for each $i\in I$ by Proposition \ref{lyap}, $\E$ possesses its extended economy $\E^\diamond=\{ L^\infty(\Omega,\Sigma,\mu),\S,(\succsim_i^\diamond,\chi_{\Omega_i})_{i=1}^n \}$ with an integral transformation determined by the formula \eqref{eq3}. Then as demonstrated in \cite[Lemma 3.2]{hs13}, by the Banach--Alaoglu theorem \cite[Corollary V.4.4]{ds58}, $\A$ is weakly$^*\!$ compact in the $n$-fold product space $[L^\infty(\Omega,\Sigma,\mu)]^n$ of $L^\infty(\Omega,\Sigma,\mu)$. Take any $(f_1,\dots,f_n)\in \A$ and let 
$$
\A{(f_1,\dots,f_n)}:=\left\{ (g_1,\dots,g_n)\in \A\mid g_i\sim_i^\diamond f_i \ \forall i\in I \right\}.
$$
Since each $\succsim_i^\diamond$ is continuous, the set $\A{(f_1,\dots,f_n)}$ is nonempty and weakly$^*\!$ compact in $[L^\infty(\Omega,\Sigma,\mu)]^n$. According to the Krein--\hspace{0pt}Milman theorem \cite[Lemma V.8.2]{ds58}, $\A{(f_1,\dots,f_n)}$ has an extreme point $(g_1,\dots,g_n)$. We claim that each of $g_i$ is a characteristic function. Suppose, to the contrary, that $g_j$ is not a characteristic function for some $j$. By virtue of the fact that $\sum_{i=1}^ng_i=1$ and $g_i\ge 0$ for each $i\in I$, we may assume without loss of generality that there exist $\varepsilon>0$ and $A\in \Sigma$ with $\mu(A)>0$ such that $\varepsilon<g_1,g_2<1-\varepsilon$ on $A$. Define the $\mu$-continuous vector measure $m:\Sigma\to \R\times \prod_{i=1}^nE_i$ by $m:=(\mu,m_1,\dots,m_n)$. It follows from Proposition \ref{lyap} that there exists a measurable subset $B\subset A$ such that $m(B)=m(A)/2$. Set $h=\varepsilon(\chi_A-2\chi_B)$. Then $h\ne 0$, $0\le g_1\pm h,g_2\pm h\le 1$, and $\int hdm=0$. Since $\varphi_i(\int(g_i\pm h)dm_i)=\varphi_i(\int g_idm_i)=\varphi_i(\int f_idm_i)$ for $i=1,2$, we have $(g_1\pm h,g_2\mp h,g_3,\dots,g_n) \in \A{(f_1,\dots,f_n)}$. This yields:
$$
(g_1,\dots,g_n)=\frac{1}{2}\left[ (g_1+h,g_2-h,g_3,\dots,g_n)+(g_1-h,g_2+h,g_3,\dots,g_n) \right],
$$ 
which means that $(g_1,\dots,g_n)$ is a convex combination of two distinct elements $(g_1+h,g_2-h,g_3,\dots,g_n)$ and $(g_1-h,g_2+h,g_3,\dots,g_n)$ in $\A{(f_1,\dots,f_n)}$, an obvious contradiction to the fact that $(g_1,\dots,g_n)$ is an extreme point in $\A{(f_1,\dots,f_n)}$. 
\end{proof}

\begin{cor}
If $\E$ has its extended economy $\E^\diamond$ with an integral transformation, then the inclusions $\P_\mathrm{PO}\subset \A_\mathrm{PO}$ and $\P_\mathrm{C}\subset \A_\mathrm{C}$ hold whenever $(\Omega,\Sigma,\mu)$ is saturated. 
\end{cor}

\begin{proof}
To demonstrate the first inclusion, take any Pareto optimal partition $(A_1,\dots,A_n)$ in $\E$. If $(\chi_{A_1},\dots,\chi_{A_n})$ is not a Pareto optimal allocation in $\E^\diamond$, then there exists another allocation $(f_1,\dots,f_n)$ such that $f_i\,\succsim_i^\diamond\,\chi_{A_i}$ for each $i\in I$ and $f_j\,\succ_j^\diamond\,\chi_{A_j}$ for some $j\in I$. It follows from Theorem \ref{thm4} that there exists a partition $(B_1,\dots,B_n)$ in $\E$ such that $f_i\sim_i^\diamond\chi_{B_i}$ for each $i\in I$, which yields $B_i\,\succsim_i\,A_i$ for each $i\in I$ and $B_j\,\succ_j\,A_j$ for $j$, a contradiction to the Pareto optimality of $(A_1,\dots,A_n)$ in $\E$. 

To demonstrate the second inclusion, take any core partition $(A_1,\dots,A_n)$ in $\E$. If $(\chi_{A_1},\dots,\chi_{A_n})$ is not a core allocation in $\E^\diamond$, then some coalition $S\in \I$ improves upon $(\chi_{A_1},\dots,\chi_{A_n})$, and hence, there exists an $S$-allocation $(f_1,\dots,f_n)$ such that $f_i\,\succ_i^\diamond\,\chi_{A_i}$ for each $i\in S$. It follows from Theorem \ref{thm4} that there exists a partition $(B_1,\dots,B_n)$ in $\E$ such that $f_i\sim_i^\diamond\chi_{B_i}$ for each $i\in I$, which yields $B_i\,\succ_i\,A_i$ for each $i\in S$, a contradiction to the fact that $(A_1,\dots,A_n)$ is a core partition in $\E$.
\end{proof}

The following converse of Theorem \ref{thm4} yields an intriguing characterization of saturation in terms of the interplay between the economy $\E$ with preference relations with a continuous representation in a Banach space and its extended economy $\E^\diamond$ with an integral transformation. 

\begin{thm}
A finite measure space $(\Omega,\Sigma,\mu)$ is saturated if every economy $\E=\{ (\Omega,\Sigma,\mu),(\succsim_i,\Omega_i)_{i=1}^n \}$ for which each $\succsim_i$ admits a continuous representation in an infinite-dimensional Banach space and its extended economy $\E^\diamond=\{ L^\infty(\Omega,\Sigma,\mu),\S,(\succsim_i^\diamond,\chi_{\Omega_i})_{i=1}^n \}$ with an integral transformation are such that for every allocation $(f_1,\dots,f_n)$ in $\E^\diamond$ there exists a partition $(A_1,\dots,A_n)$ in $\E$ satisfying $f_i\sim_i^\diamond\chi_{A_i}$ for each $i\in I$.  
\end{thm}

\begin{proof}
Suppose that $(\Omega,\Sigma,\mu)$ is not saturated. As noted in the proof of Theorem \ref{thm3}, for every infinite-\hspace{0pt}dimensional Banach space $E$, there exists a $\mu$-continuous vector measure $m:\Sigma\to E$ such that its range $m(\Sigma)$ is not convex. Then there exist $x,y\in m(\Sigma)$ and $t\in (0,1)$ such that $tx+(1-t)y\not\in m(\Sigma)$. Let $C,D\in \Sigma$ be such that $x=m(C)$ and $y=m(D)$ and define $f_1:=t\chi_C+(1-t)\chi_D\in \S$ and $f_j:=(n-1)^{-1}(1-f_1)\in \S$ for $j=2,3,\dots,n$. By construction, $\sum_{i=1}^nf_i=1$. Let $\varphi_1^\diamond:m(\S)\to \R$ be a continuous function with $\varphi_1^\diamond(0)=0$ that attains the maximum at the unique point $x_1:=\int f_1dm\in m(\S)$. For $j=2,3,\dots,n$, take any continuous function $\varphi_j^\diamond:m(\S)\to \R$ with $\varphi_j^\diamond(0)=0$, define the preference relation $\succsim_i^\diamond$ on $\S$ via the formula \eqref{eq3}, and denote its restriction to $\Sigma$ by $\succsim_i$ for each $i\in I$. Then $\E=\{ (\Omega,\Sigma,\mu),(\succsim_i,\Omega_i)_{i=1}^n \}$ is an economy for which each $\succsim_i$ admits a continuous representation in the common Banach space $E$ by Theorem \ref{thm1} and $\E^\diamond=\{ L^\infty(\Omega,\Sigma,\mu),\S,(\succsim_i^\diamond,\chi_{\Omega_i})_{i=1}^n \}$ is an extended economy of $\E$ with an integral transformation. Since $(f_1,\dots,f_n)$ is an allocation in $\E^\diamond$, there exists a partition $(A_1,\dots,A_n)$ in $\E$ such that $f_i\sim_i^\diamond\chi_{A_i}$ for each $i\in I$. This means that $\varphi_1^\diamond(x_1)=\varphi_1^\diamond(m(A_1))$, and hence, $x_1=tx+(1-t)y=m(A_1)\in m(\Sigma)$, a contradiction. 
\end{proof}

\subsection{Existence Results}
We present four existence results on partitions in $\E$. We first show under the saturation hypothesis the existence of individually rational Pareto optimal partitions without any convexity assumption on the preferences of each individual. 

\begin{thm}
\label{thm5}
If $(\Omega,\Sigma,\mu)$ is saturated, then for every economy $\E=\linebreak\{ (\Omega,\Sigma,\mu),(\succsim_i,\Omega_i)_{i=1}^n \}$ for which each $\succsim_i$ admits a continuous representation in a separable Banach space, there exists an individually rational Pareto optimal partition in $\E$. 
\end{thm}

\begin{proof}
It follows from Theorem \ref{thm4} that the economy $\E$ has its extended economy $\E^\diamond=\{ L^\infty(\Omega,\Sigma,\mu),\S,(\succsim_i^\diamond,\chi_{\Omega_i})_{i=1}^n \}$ with an integral transformation such that for every allocation $(f_1,\dots,f_n)$ in $\E^\diamond$, there exists a partition $(A_1,\dots,A_n)$ in $\E$ satisfying $f_i\sim_i^\diamond\chi_{A_i}$ for each $i\in I$. Since each $\succsim_i^\diamond$ is represented by the form \eqref{eq3}, where $m_i:\Sigma \to E_i$ is a $\mu$-continuous vector measure with values in a separable Banach space $E_i$ and $\varphi_i^\diamond:m_i(\S)\to \R$ is a continuous function, and $\A$ is weakly$^*\!$ compact in $[L^\infty(\Omega,\Sigma,\mu)]^n$, so is $\A_{\text{IR}}$. Then there exists a solution to the maximization problem:
\begin{equation}
\label{P}
\max\left\{ \sum_{i=1}^n\varphi_i^\diamond\left( \int_\Omega f_i(\omega)dm_i \right) \mid (f_1,\dots,f_n)\in \A_{\text{IR}} \right \} \tag{P}.
\end{equation}
The individual rationality of a solution $(f_1,\dots,f_n)$ to \eqref{P} is obvious. If $(f_1,\dots,f_n)$ is not Pareto optimal, then there is another allocation $(g_1,\dots,g_n)$ such that $g_i\succsim_i^\diamond f_i$ for each $i\in I$ and $g_j\succ_j^\diamond f_j$ for some $j\in I$. This means that $\sum_{i=1}^n\varphi_i^\diamond(\int g_idm_i)>\sum_{i=1}^n\varphi_i^\diamond(\int f_idm_i)$, a contradiction to the fact that $(f_1,\dots,f_n)$ is a solution to \eqref{P}. Therefore, any partition $(A_1,\dots,A_n)$ satisfying $f_i\sim_i^\diamond\chi_{A_i}$ for each $i\in I$ is an individually rational Pareto optimal partition in $\E$. 
\end{proof}

For the existence of core partitions, the convexity of preferences of each individual are imposed additionally. 

\begin{thm}
If $(\Omega,\Sigma,\mu)$ is saturated, then for every economy $\E=\linebreak\{ (\Omega,\Sigma,\mu),(\succsim_i,\Omega_i)_{i=1}^n \}$ for which each $\succsim_i$ admits a continuous convex representation in a separable Banach space, there exists a core partition in $\E$. 
\end{thm}

\begin{proof}
It follows from Theorem \ref{thm4} that the economy $\E$ has its extended economy $\E^\diamond=\{ L^\infty(\Omega,\Sigma,\mu),\S,(\succsim_i^\diamond,\chi_{\Omega_i})_{i=1}^n \}$ with a quasiconcave integral transformation such that for every allocation $(f_1,\dots,f_n)$ in $\E^\diamond$, there exists a partition $(A_1,\dots,A_n)$ in $\E$ satisfying $f_i\sim_i^\diamond\chi_{A_i}$ for each $i\in I$. Since each $\succsim_i^\diamond$ is represented by the form \eqref{eq3}, where $m_i:\Sigma \to E_i$ is a $\mu$-continuous vector measure with values in a separable Banach space $E_i$ and $\varphi_i^\diamond:m_i(\S)\to \R$ is a continuous quasiconcave function, the utility function $\nu_i^\diamond:\S\to \R$ defined by $\nu_i^\diamond(f):=\varphi_i^\diamond(\int fdm_i)$ is weakly$^*\!$ continuous and quasiconcave on $\S$. 

The market game $V:\I\to 2^{\R^n}$ with nontransferrable utility (NTU) for the extended economy $\E^\diamond$ is a set-valued mapping defined by:
\[
V(S)=\left\{ (x_1,\dots,x_n)\in \R^n \left| 
\begin{array}{ll}
\exists\,\text{$S$-\hspace{0pt}allocation $(f_1,\dots,f_n)\in \A$:} \\
\text{$x_i\le \nu_i^\diamond(f_i)$ $\forall i\in S$} 
\end{array}
\right.\right\}.
\]
By construction, $V(S)$ is the subset of the utility possibility set of the individuals in which payoff vectors are attainable via some coalition $S\in \I$. The \textit{core} $C(V)$ of the NTU game $V$ is given by: 
\[
C(V)=\left\{ (x_1,\dots,x_n)\in V(I) \mid \not\exists (S,y)\in \I\times V(S): x_i<y_i\ \forall i\in S \right \}.
\]
By the celebrated theorem of \citet{sc67}, $C(V)$ is nonempty if $V$ is comprehensive below and balanced, $V(S)$ is closed and bounded from above for every $S\in \I$, and $x=(x_1,\dots,x_n)\in \R^n$, $y=(y_1,\dots,y_n)\in V(S)$ and $x_i=y_i$ for each $i\in S$ imply $x\in V(S)$. We show that $V$ satisfies these conditions. 

It is easy to see that each $V(S)$ is comprehensive from below, i.e., $x=(x_1,\dots,x_n)\in \R^n$, $y=(y_1,\dots,y_n)\in V(S)$ and $x_i\le y_i$ for each $i\in I$ imply $x\in V(S)$. Since each $\nu_i^\diamond$ is weakly$^*\!$ continuous, and hence, bounded on the weakly$^*\!$ compact set $\S$, for each $S\in \I$ there exists $M_S\in \R$ such that $x_i\le M_S$ for every $x\in V(S)$ and $i\in S$.  

We shall show that $V$ is a balanced game. To this end, let $\B$ be a balanced family of $\I$ with balanced weights $\{\lambda^S\ge 0 \mid S\in \B \}$ and let $\B_i=\{ S\in \B \mid i\in S \}$. We then have $\sum_{S\in \B_i}\lambda^S=1$ for each $i\in I$. Define:
\[
\chi_i^S=
\begin{cases}
1 & \text{if $S\in \B_i$}, \\
0 & \text{otherwise}
\end{cases}
\quad\text{and} \quad t^S=\frac{1}{n}\sum_{i\in I}\lambda^S\chi_i^S.
\]
Then, we have:
$$
\sum_{S\in \B}t^S=\frac{1}{n}\sum_{S\in \B}\sum_{i\in I}\lambda^S\chi_i^S=\frac{1}{n}\sum_{i\in I}\sum_{S\in \B_i}\lambda^S=1.
$$
Choose any $x=(x_1,\dots,x_n)\in \bigcap_{S\in \B}V(S)$. Then, for every $S\in \B$, there exists an $S$-\hspace{0pt}allocation $(f_1^S,\dots,f_n^S)$ such that $x_i\le \nu_i^\diamond(f_i^S)$ for each $i\in S$. Let $f_i=\sum_{S\in \B}t^Sf_i^S$ for each $i\in I$. Then, $(f_1,\dots,f_n)$ is an allocation because $\A$ is convex. Since $x_i\le \nu_i^\diamond(f_i)$ for each $i\in I$ by the quasiconcavity of $\nu_i^\diamond$, we have $x\in V(I)$. Therefore, $\bigcap_{S\in \B}V(S)\subset V(I)$, and consequently, $V$ is balanced. 

We finally show that $V(S)$ is closed for every $S\in \B$. Let $\{ x^k \}_{k\in \N}$ be a sequence in $V(S)$ converging to $x\in \R^n$. Then, there exists an allocation $(f_1^k,\dots,f_n^k)$ such that $x_i^k\le \nu_i^\diamond(f_i^k)$ for each $i\in S$ and $k\in \N$. Since $\A$ is weakly$^*\!$ compact, the sequence $\{ (f_1^k,\dots,f_n^k) \}_{k\in \N}$ contains a subsequence that is weakly$^*\!$ convergent to $(f_1,\dots,f_n)\in \A$. Then we have $x_i\le \nu_i^\diamond(f_i)$ for each $i\in S$ by the weak$^*\!$ continuity of $\nu_i^\diamond$. It is easy to verify that $(f_1,\dots,f_n)$ is an $S$-\hspace{0pt}allocation. Thus, we obtain $x\in V(S)$, and hence, $V(S)$ is closed. 

Since $C(V)$ is nonempty, one can choose an element $(x_1,\dots,x_n)$ in $C(V)$. Then there exists an $I$-\hspace{0pt}allocation $(f_1,\dots,f_n)$ such that $x_i\le \nu_i^\diamond(f_i)$ for each $i\in I$. Suppose that $(f_1,\dots,f_n)$ is not a core allocation in $\E^\diamond$. Then there exists an $S$-\hspace{0pt}allocation $(g_1,\dots,g_n)$ such that $\nu_i^\diamond(f_i)<\nu_i^\diamond(g_i)$ for each $i\in S$. Then we have $(\nu_1^\diamond(g_1),\dots,\nu_n^\diamond(g_n))\in V(S)$ and $x_i<\nu_i^\diamond(g_i)$ for each $i\in S$, which contradicts the fact that $(x_1,\dots,x_n)$ is in $C(V)$. Let $(x_1,\dots,x_n)$ be in $C(V)$ and define the set $\A(x_1,\dots,x_n)$ by:
$$
\A(x_1,\dots,x_n)=\left\{ (f_1,\dots,f_n)\in \A\mid x_i\le \nu_i^\diamond(f_i) \ \forall i\in I \right\}. 
$$
It follows from the above argument that $\A(x_1,\dots,x_n)$ is a nonempty subset of $\A_\mathrm{C}$. Since $\A(x_1,\dots,x_n)$ is nonempty and weakly$^*\!$ compact in $[L^\infty(\Omega,\Sigma,\mu)]^n$, according to the Krein--\hspace{0pt}Milman theorem, $\A(x_1,\dots,x_n)$ has an extreme point $(g_1,\dots,g_n)$. Precisely in the same way with the proof of Theorem \ref{thm4}, we can show that each $g_i$ is a characteristic function. Therefore, there exists a partition $(A_1,\dots,A_n)$ of $\Omega$ such that $(\chi_{A_1},\dots,\chi_{A_n})\in  \A(x_1,\dots,x_n)$. This means that $(A_1,\dots,A_n)$ is a core partition in $\E$. 
\end{proof}

For the existence of Walrasian equilibrium partitions with a positive price, the monotonicity of preferences of each individual are imposed further. 

\begin{thm}
\label{thm6}
If $(\Omega,\Sigma,\mu)$ is saturated, then for every economy $\E=\linebreak\{ (\Omega,\Sigma,\mu),(\succsim_i,\Omega_i)_{i=1}^n \}$ for which each $\succsim_i$ admits a continuous convex monotone representation in a separable Banach space, there exists a Walrasian equilibrium partition in $\E$ with a positive price. 
\end{thm}

\begin{proof}
By Proposition \ref{lyap} and Theorem \ref{thm1}, the economy $\E$ has its extended economy $\E^\diamond=\{ L^\infty(\Omega,\Sigma,\mu),\S,{(\succsim_i^\diamond,\chi_{\Omega_i})}_{i=1}^n \}$ with a quasiconcave integral transformation such that each $\succsim_i^\diamond$ is represented by the form \eqref{eq3}, where $m_i:\Sigma \to E_i$ is a vector measure with values in a separable Banach space $E_i$ and $\varphi_i^\diamond:m_i(\S)\to \R$ is a continuous quasiconcave function. Since $\S$ is weakly$^*\!$ closed in $L^\infty(\Omega,\Sigma,\mu)$ and each utility function defined by $\nu_i^\diamond(f)=\varphi_i^\diamond(\int fdm_i)$ is weakly$^*\!$ continuous and quasiconcave on $\S$ satisfying \cite[Monotonicity Assumption]{be72}, and hence, all the assumptions of \cite[Theorem 1]{be72} are met for $\E^\diamond$. Then there exists a Walrasian equilibrium $(\hat{p},(\hat{f}_1,\dots,\hat{f}_n))$ with positive price $\hat{p}$ in $\E^\diamond$. Define the set of Walrasian allocations associated with the equilibrium price $\hat{p}$ by:   
$$
\A_\mathit{W}(\hat{p})=\left\{ (f_1,\dots,f_n)\in \A \left| \begin{array}{l}\text{$f_i$ is a maximal element for $\succsim_i$} \\ \text{on $B_i^\diamond(\hat{p},\chi_{\Omega_i})$ $\forall i\in I$} \end{array} \right.\right\}.
$$
Then $\A_\mathit{W}(\hat{p})$ is nonempty and weakly$^*\!$ compact in $[L^\infty(\Omega,\Sigma,\mu)]^n$. According to the Krein--\hspace{0pt}Milman theorem, $\A_\mathit{W}(\hat{p})$ has an extreme point $(g_1,\dots,g_n)$. Precisely in the same way as the proof of Theorem \ref{thm4}, we can show that each $g_i$ is a characteristic function. Therefore, there exists a partition $(\hat{A}_1,\dots,\hat{A}_n)$ of $\Omega$ such that $(\chi_{\hat{A}_1},\dots,\chi_{\hat{A}_n})\in  \A_\mathit{W}(\hat{p})$. This means that the price-\hspace{0pt}partition pair $(\hat{p},(\hat{A}_1,\dots,\hat{A}_n))$ is a Walrasian equilibrium with positive price $\hat{p}$ in $\E$. 
\end{proof}

Under the same assumption as Theorem \ref{thm6} the existence of Pareto optimal envy-free partitions is guaranteed. 

\begin{thm}
If $(\Omega,\Sigma,\mu)$ is saturated, then for every economy $\E=\linebreak\{ (\Omega,\Sigma,\mu),(\succsim_i,\Omega_i)_{i=1}^n \}$ for which each $\succsim_i$ admits a continuous convex monotone representation in a separable Banach space, there exists a Pareto optimal envy-free partition in $\E$. 
\end{thm}

\begin{proof}
Let $\E^\diamond=\{ L^\infty(\Omega,\Sigma,\mu),\S,{(\succsim_i^\diamond,\chi_{\Omega_i})}_{i=1}^n \}$ be an extended economy of $\E$ with a quasiconcave integral transformation such that for every allocation $(f_1,\dots,f_n)$ in $\E^\diamond$, there exists a partition $(A_1,\dots,A_n)$ in $\E$ satisfying $f_i\sim_i^\diamond\chi_{A_i}$ for each $i\in I$. Consider the extended economy $\overline{\E^\diamond}=\{ L^\infty(\Omega,\Sigma,\mu),\S,\linebreak {(\succsim_i^\diamond)}_{i=1}^n,n^{-1}\chi_\Omega \}$ in which each individual possesses the equal initial endowment $n^{-1}\chi_\Omega\in \S$. Then a Walrasian equilibrium $(\bar{p},(\bar{f}_1,\dots,\bar{f}_n))$ with positive price $\bar{p}$ in $\overline{\E^\diamond}$ exists as proved in Theorem \ref{thm6}. The Walrasian allocation $(\bar{f}_1,\dots,\bar{f}_n)$ is Pareto optimal in $\E^\diamond$ (because of the the first welfare theorem) and in view of the equal budget set among all individuals, it is evidently envy free in $\E^\diamond$. Take any partition $(\bar{A}_1,\dots,\bar{A}_n)$ in $\E$ satisfying $\bar{f}_i\sim_i^\diamond\chi_{\bar{A}_i}$ for each $i\in I$. If $(\bar{A}_1,\dots,\bar{A}_n)$ is not Pareto optimal in $\E$, then there exists another partition $(A_1,\dots,A_n)$ such that $A_i\succsim_i \bar{A}_i$ for each $i\in I$ and $A_j\succ_j \bar{A}_j$ for some $j\in I$. This means that $\chi_{\bar{A}_i}\succsim_i^\diamond \bar{f}_i$ or each $i\in I$ and $\chi_{\bar{A}_j}\succsim_j^\diamond \bar{f}_j$ for $j$, a contradiction to the Pareto optimality of $(\bar{f}_1,\dots,\bar{f}_n)$ in $\E^\diamond$. Hence, $(\bar{A}_1,\dots,\bar{A}_n)$ is a Pareto optimal partition in $\E$. If $(\bar{A}_1,\dots,\bar{A}_n)$ is not envy free in $\E$, then $\bar{A}_j\,\succ_i\,\bar{A}_i$ for some $i\ne j$. This implies that $\bar{f}_j\,\succ_i^\diamond\,\bar{f}_i$, a contradiction to the fact that $(\bar{f}_1,\dots,\bar{f}_n)$ is envy free in $\E^\diamond$. 
\end{proof}

\begin{rem}
To obtain a countably additive equilibrium price in Theorem \ref{thm6}, the same argument as in \cite{be72}, which uses the Yosida--Hewitt decomposition in $\mathit{ba}(\Sigma,\mu)$, is valid. In particular, if the extended economy $\E^\diamond$ is such that each $\succsim_i^\diamond$ is monotone and has a Mackey continuous extension to the positive cone of $L^\infty(\Omega,\Sigma,\mu)$ preserving monotonicity, then \cite[Theorem 2]{be72} is applicable to the proof of Theorem \ref{thm6}, and thereby, $\hat{p}$ can be taken in $L^1(\Omega,\Sigma,\mu)$. 
\end{rem}

\begin{rem}
For an economy $\E$ in which each $\succsim_i$ admits a continuous representation in $\R^{l_i}$, the existence of Walrasian equilibria (see \cite{hu11}), the nonemptiness of the core (see \cite{hu08,sa06}), the nonemptiness of the fuzzy core and the existence of supporting prices (see \cite{hs12}), and the existence of Pareto optimal $\alpha$-fair partitions (see \cite{sa11}) were established under the convexity hypothesis on $\succsim_i$.  Under the alternative continuity hypothesis on $\succsim_i$, the existence of Pareto optimal envy-free partitions was given in \cite{hs13} without any convexity hypothesis on $\succsim_i$. For other solution concepts regarding the fair division problems, see \cite{sa08,sa11,sv10b,sv11}. 
\end{rem}

\section{Concluding Remarks}
We conclude this paper by raising an open problem. We leave aside the existence of supporting prices for Pareto optimal partitions under the convexity assumption. The existence of supporting prices in $L^1(\Omega,\Sigma,\mu)$ was obtained in \cite{hs12} with a finite-dimensional vector representation of preference relations. The problem is a standard application of the separation theorem in $L^\infty(\Omega,\Sigma,\mu)$. It should be noted that thanks to Theorem \ref{thm4}, supporting prices in an extended economy are automatically the ones in the original economy. As demonstrated in Theorem \ref{thm5}, however, for the existence of Pareto optimal partitions, the convexity assumption is unnecessary. Thus, the challenging problem is instead to demonstrate the second fundamental theorem of welfare economics for economies without any convexity assumption. For nonconvex production economies with an infinite-dimensional commodity space, \cite{kv88} derived a price system in the Clarke normal cone that is consistent with a Pareto optimal allocation, and their result covers the commodity space of $L^\infty(\Omega,\Sigma,\mu)$. It is well-known that in the presence of nonconvexity Clarke normal cones are strictly larger than \citeauthor{mo00} normal cones (see \cite{mo06a}), so that price systems in the latter are called for to derive sharper necessary conditions for Pareto optimality. It is \cite{kh91} who first introduced \citeauthor{mo00} normal cones to obtain the second welfare theorem into finite-dimensional nonconvex production economies. For the extension of \cite{kh91} to Banach or Asplund spaces of commodities, see \cite{bm10,mm01,mo00,mo05,mo06b}. The second welfare theorem without convexity assumptions specific to optimal partitions based on \citeauthor{mo00} normal cones is still unknown.

\end{document}